\documentclass[preprint,10pt,3p]{elsarticle}
\usepackage{amsmath,amsthm,amsfonts,amssymb}
\usepackage[bookmarksnumbered, colorlinks, plainpages]{hyperref}
\usepackage{xcolor}
\usepackage{sectsty}
\usepackage{comment}
\usepackage{tabularx,booktabs}

\usepackage{natbib}
\definecolor{astral}{RGB}{46,116,181}
\sectionfont{\color{astral}}
\newtheorem{theorem}{Theorem}[section]

\newtheorem{corollary}{Corollary}[section]
\newtheorem{proposition}{Proposition}[section]

\usepackage{hyperref}
\journal{Elsevier}

\newcommand{\ro}[1]{\textcolor{black}{#1}} 
\newcommand{\rt}[1]{\textcolor{black}{#1}}

\definecolor{fgreen}{rgb}{0.13, 0.55, 0.13}
\newcommand{\fb}[1]{\textcolor{black}{#1}} 

\newcommand{\ra}[1]{\mathrm{rank}({#1})}  

\usepackage{amsmath}
\usepackage{amsfonts}
\usepackage{amssymb,enumerate}
\usepackage{color}
\usepackage{xcolor}

\def\k{{\rm{ker}}}

\def\rg{{\rm{ran}}}
\def\n{{\rm{null}}}

\numberwithin{equation}{section}
\begin{document}
\begin{frontmatter}\title{On the invertibility of matrices with a double saddle-point structure}
	
	\author[add1,add3]{Fatemeh P. A. Beik}
	\ead{f.beik@vru.ac.ir}
	\author[add2]{Chen Greif}
	\ead{greif@cs.ubc.ca}
	 \author[add3]{Manfred Trummer}
	 \ead{trummer@sfu.ca}

\address[add1]{Department of Mathematics, Vali-e-Asr University of Rafsanjan, P.O. Box 518, Rafsanjan, Iran}
\address[add2]{Department of Computer Science, The University of British Columbia, Vancouver, Canada V6T 1Z4}
\address[add3]{Department of Mathematics, Simon Fraser University, Burnaby, BC V5A 1S6, Canada}
	
	\begin{abstract}
We establish necessary and sufficient conditions for invertibility of symmetric three-by-three block matrices having a double saddle-point structure \fb{that guarantee the unique solvability of double saddle-point systems}. We consider various scenarios, including the case where all diagonal blocks are allowed to be rank deficient. Under certain conditions related to the nullity of the blocks and intersections of their kernels, an explicit formula for the inverse is derived.     
	\end{abstract}
	

\end{frontmatter}

\section{Introduction} \label{sec:intro}

We consider $\ell \times \ell$ double saddle-point matrices of the form
\begin{equation}
	\mathcal{K} = \begin{bmatrix}
			A&{{B^T}}&0\\
			B&{ - D}&{{C^T}}\\
			0&C&E
	\end{bmatrix},
\label{eq:coef}
\end{equation}
where $A\in \mathbb{R}^{n\times n},D\in \mathbb{R}^{m\times m}$, $E\in \mathbb{R}^{p\times p}$ and $\ell=n+m+p$. Given $b \in \mathbb{R}^{\ell}$, these matrices and the corresponding linear systems with solution vector $u\in \mathbb{R}^{\ell}$,
\begin{eqnarray}
\mathcal{K}u=b,
\label{eq:system}
\end{eqnarray} 
arise in a variety of applications in computational science and engineering, and their numerical solution has been the subject of much interest and investigation in recent years. Block-tridiagonal linear systems of equations of the form \eqref{eq:system} arise in the finite element or finite difference discretization of the coupled Stokes--Darcy flow equations \cite{BB2022, Cai, Chid,GH2023},
the treatment of mixed and mixed-hybrid formulations of second--order elliptic equations, elasticity, and liquid crystal problems \cite{BB2018,BB20182,Boffi,Maryska,Ramage}, poromechanical equations \citep{ffjct2019},  PDE-constrained optimization problems \citep{pw2012, rdw2010}, and several other important applications. 

The leading block of $\mathcal{K}$ is often symmetric positive definite or symmetric positive semidefinite. The matrices $D$ or $E$ or both may be zero, depending on the application. Permutations may lead to a different block structure of the matrix and additional considerations in the design of numerical solvers, depending among other factors on the dimensions of the blocks and their ranks. 

The main diagonal blocks $A$, $D$, and $E$ in \eqref{eq:coef} are assumed to be symmetric positive (semi)definite in most cases. This implies that the matrix $\mathcal{K}$ is indefinite, and numerically solving the system \eqref{eq:system} may be challenging, especially if the matrix is large and sparse, and iterative methods \cite{SaadBook}  are required.  \rt{It is worth mentioning  that,  under certain assumptions on the block matrices, several preconditioning techniques have been proposed  to accelerate the convergence of Krylov subspace methods  for solving the double (multiple) saddle-point linear system \eqref{eq:system} or its permuted forms, see \cite{Balani2,Balani,BB2018,BB2022,Cai,Chid,Pearson,Song} and the references therein. Our present study, however, does not focus on developing preconditioning techniques. Here,} we are interested in investigating conditions on invertibility of matrices of the form \eqref{eq:coef}, \fb{which ensure the existence of a unique solution for \eqref{eq:system}}. We will make nonrestrictive  assumptions on the rank structure of the blocks of the matrix $\mathcal{K}$, and study invertibility of $\mathcal{K}$ given by \eqref{eq:coef}. 

Some results on invertibility of double saddle-point matrices exist in the literature, but to the best of our knowledge they are more limited in scope than the results we present in this paper. In \cite[Propositions 2.1--2.3]{XL2020}, considering the specific situation where $D$ and $E$ are both zero matrices, the authors provide some conditions for invertibility, based either on assuming full row rank of $B$ and $C$ or assuming zero-only intersections of the kernels of some of the blocks. Some additional conditions on invertibility are provided in \cite{BG2022}, where all diagonal blocks are assumed to be potentially nonzero. 

An outline of the remainder of the paper follows. 
In section \ref{sec:inv}, we study the nonsingularity of $\mathcal{K}$ under the assumption that all  three block diagonal matrices $A,D$, and $E$ are possibly rank deficient.
We further focus on the case where the leading block $A$ has a specific nullity in section \ref{sec3}, where an explicit formula for $\mathcal{K}^{-1}$ is also derived.
Finally,  we draw some conclusions in Section \ref{sec4}. \\

\noindent  {\bf Notation.} Let $W$ be a square matrix. 
We use the notation $W\succ(\succeq)~0$ when $W$ is a symmetric (semi)-positive definite matrix. Given a matrix $M$, \fb{the range and kernel of $M$ are respectively denoted by $\rg(M)$ and $\k(M)$.  The notations $\ra{M}$ and $\n(M)$ stand for the dimensions of $\rg(M)$ and $\k(M)$, respectively.}
Given vectors $x$,  $y$ and $z$ of dimensions $n$, $m$ and $p$, we use {\textsc {Matlab}} notation $[x;y;z]$ to denote a column vector of dimension $n+m+p$.

\section{Necessary and sufficient conditions on invertibilty}\label{sec:inv}

In this section we present necessary and sufficient conditions for the invertibility of  $\mathcal{K}$ under various assumptions including the case where the matrices $A$, $D$ and $E$ are allowed to be singular simultaneously. We first recall a result that provides necessary conditions for nonsingularity of $\mathcal{K}$ with respect to its blocks {and sufficient conditions for the invertibilty of $\mathcal{K}$} under stricter assumptions.
\begin{proposition}{\rm \cite[Proposition 2.1]{BG2022}}\label{prop1.1}
	The following conditions are necessary for $\mathcal{K}$ to be invertible:
\begin{itemize}
	\item [$(i)$] $\k(A) \cap \k(B) = \{0\}$;
	\item [$(ii)$] $\k(B^T) \cap \k(D)\cap \k(C)= \{0\}$;
	\item [$(iii)$] $\k(C^T) \cap \k(E) = \{0\}$.
\end{itemize}
In the case where $A$ is nonsingular, sufficient conditions for
$\mathcal{K}$ to be invertible are that $S_1=D+BA^{-1}B^T$ and $S_2=E+C\ro{S}_1^{-1}C^T$ are invertible.
\end{proposition}

We now extend the results of Proposition~\ref{prop1.1}. In particular, it turns out that the existence of $\mathcal{K}^{-1}$
can be concluded without checking the invertibility of  $S_1=D+BA^{-1}B^T$ and $S_2=E+C\ro{S}_1^{-1}C^T$ (which may not even be  defined if $A$ is singular), provided that certain conditions on the blocks of $\mathcal{K}$ hold. In the theorems that follow, we show that \rt{the kernel of $\mathcal{K}$ is trivial} to establish the existence of   $\mathcal{K}^{-1}$.
 
\begin{theorem}
Let $\mathcal{K}$ be given by \eqref{eq:coef} such that $A\succeq 0,D\succeq 0, E\succeq 0$ and 
\begin{equation}\label{ker3}
\mathbb{K}:=\k(B^T) \cap \k(D)\cap \k(C)=\{0\}.
\end{equation}
Then, the following statements hold:
\begin{enumerate}[$(1)$]
	\item If $A\succ 0$ and $\k(C^T) \cap \k(E) =\{0\}$, then $\mathcal{K}$ is invertible;
	\item If $E\succ 0$ and $\k(A) \cap \k(B)=\{0\}$, then $\mathcal{K}$ is invertible;
	\item If $\rg(B) \cap \rg(C^T)=\{0\},$ $\k(C^T) \cap \k(E) =\{0\}$ and $\k(A) \cap \k(B)=\{0\}$, then $\mathcal{K}$ is invertible.
\end{enumerate}	
\label{thm1}
\end{theorem}
\begin{proof}
	Let $\mathcal{K}\bar u=0$ where $\bar u=[ x;  y;  z]$. As a result, we have		
	\begin{subequations}\label{hom}
		\begin{align}
			Ax+B^Ty~~~~~~~~& = 0\label{hom1} ; \\
			Bx-Dy+C^Tz & = 0\label{hom2} ; \\
			Cy+Ez~&= 0. \label{hom3}
		\end{align}
	\end{subequations}
Multiplying \eqref{hom2} on the left by $y^T$, we get 
\begin{equation}\label{eq1.3}
y^TBx-y^TDy+y^TC^Tz=0.
\end{equation}
Form Eqs.\  \ro{\eqref{hom1} and \eqref{hom3}}, we respectively obtain 
\[
B^Ty=-Ax \quad \text{and} \quad Cy=-Ez.
\]
Substituting the above relations into \eqref{eq1.3}, we obtain
\begin{equation}\label{eq1.4}
	x^TAx+y^TDy+z^TEz=0.
\end{equation}
The semidefiniteness of $A,D$, and $E$ yields
\[
x\in \k(A),~y\in \k(D),~  \text{and\ }  z\in \k(E).
\]
From \eqref{hom1} and $x\in \k(A)$, we deduce that $y\in \k(B^T)$. Also, Eq.\  \eqref{hom3} together with $z\in \k(E)$ imply that $y \in \k(C)$. Hence, we conclude that $y\in \mathbb{K}$ where $\mathbb{K}$ is defined by \eqref{ker3}. By the assumption \eqref{ker3}, this ensures  that $y$ is a zero vector. Consequently, Eqs.\  \eqref{hom} reduce to 
	\begin{equation}
	\label{homn2}
		Bx+C^Tz =0.
	\end{equation}

\fb{To conclude the proof, we consider three cases
corresponding to statements (1)--(3) in the theorem.}

\begin{itemize}
	\item [Case (1).] If $A\succ 0$, then $x=0$ by \eqref{eq1.4}. Hence, \eqref{eq1.4} and \eqref{homn2} imply that $z\in \k(C^T) \cap \k(E)=\{0\}$. Therefore, we can conclude that statement ($1$) of the theorem holds. 
	\item [Case (2).] If  $E\succ 0$, by Eq.\  \eqref{eq1.4}, then $z=0$, \ro{which implies} that $x\in \k(A) \cap \k(B)=\{0\}$. Now it is immediate to  deduce the statement ($2$).
	\item [Case (3).] By Eq.\  \eqref{homn2}, we have $Bx=-C^Tz.$
	This says that $Bx\in \rg(C^T)$ and $C^Tz\in \rg(B)$. Consequently, the vectors $Bx$ and $C^Tz$ belong to $ \rg(B) \cap \rg(C^T)$,
	which is a trivial subspace by the assumption in $(3)$. This implies that $Bx$ and $C^Tz$ are both zero, i.e., $x\in \k(B)$ and $z\in \k(C^T)$.
	As a result, from \eqref{eq1.4} it follows that
	$x\in \k(A) \cap \k(B)$ and $z\in \k(C^T) \cap \k(E)$, which completes the proof of \ro{$(3)$}.
\end{itemize}
\end{proof}

\begin{corollary}\label{cor2.1}
	Suppose that $D, E \succ 0$, $A=0$, and $m \ge n$. Then the matrix $\mathcal{K}$ is invertible \ro{if and only if} $\ra B=n$.
\end{corollary}

\begin{proof}
Let $\ra B=n$ and $\mathcal{K}\bar u=0$ where $\bar u=[x;y;z]$.	In view of Eq.\  \eqref{eq1.4}, one can verify that $y$ and $z$ are both zero when $D, E \succ 0$. Therefore, since $A=0$, Eqs.\  \eqref{hom} reduce to $Bx=0$, which implies that $x=0$ and completes the proof of the nonsingularity of $\mathcal{K}$.

Conversely, assume that $\mathcal{K}$ is invertible. If $\ra B<n$, then there exists $x\ne 0$ such that $Bx=0$. Consequently, we have 
$\mathcal{K}\tilde u=0$ for $\tilde u=[x;0;0]$ which contradicts the assumed nonsingularity of $\mathcal{K}$. Therefore, we conclude $\ra B=n$.
\end{proof}
The following two additional corollaries can be readily proven in a similar fashion to Corollary \ref{cor2.1}; their proofs are omitted.

\begin{corollary}
	Suppose that $A,D\succ 0$ and $E=0$, and suppose further that $m\ge p$. The matrix $\mathcal{K}$ is invertible \ro{if and only if}  $\ra C=p$.
\end{corollary}

\begin{corollary}
	Suppose that $A,E\succ 0$ and $D=0$. The matrix $\mathcal{K}$ is invertible \ro{if and only if}  $\k(B^T) \cap \k(C) =\{0\}$.
\end{corollary}

We now further establish necessary and sufficient conditions on nonsingularity of $\mathcal{K}$  under different assumptions, using summation of subspaces. Recall that for given subspaces $\ro{\mathbb{S}}_1$ and $\ro{\mathbb{S}}_2$ of the vector space  over real numbers, the sum of $\ro{\mathbb{S}}_1$ and $\ro{\mathbb{S}}_2$ is the subspace
	\[
	\ro{\mathbb{S}}_1+\ro{\mathbb{S}}_2={\rm Span}\{\ro{\mathbb{S}}_1 \cup \ro{\mathbb{S}}_2\}=\{x+y~|~x\in \ro{\mathbb{S}}_1,~ y\in \ro{\mathbb{S}}_2\}.
	\]
	If $\ro{\mathbb{S}}_1 \cap \ro{\mathbb{S}}_2$ is trivial, the sum of $\ro{\mathbb{S}}_1$ and $\ro{\mathbb{S}}_2$ is called a direct sum and it is written as
	$\ro{\mathbb{S}}_1 \oplus \ro{\mathbb{S}}_2$; every $z \in \ro{\mathbb{S}}_1 \oplus \ro{\mathbb{S}}_2$ can be written as $z = x + y$ with $x \in \ro{\mathbb{S}}_1$ and $y \in \ro{\mathbb{S}}_2 $ in a unique way \cite[Subsection 0.1.3]{Horn1}. \ro{The assumptions below} are motivated by the discussion in \cite{EG2015} for the nonsingular saddle-point system
\[\underbrace {\left[ {\begin{array}{*{20}{c}}
			A&{{B^T}}\\
			B&0
	\end{array}} \right]}_{\fb{\hat{\bar{\mathcal{K}}}}}\left[ {\begin{array}{*{20}{c}}
		u\\
		p
\end{array}} \right] = \left[ {\begin{array}{*{20}{c}}
		f\\
		g
\end{array}} \right], \]
in which $A\in \mathbb{R}^{n\times n},$ $B\in \mathbb{R}^{m\times n}$, and the matrix $A$ is assumed to be a {\it maximally rank deficient} symmetric positive semidefinite matrix, i.e., 
$$ \ra A = n-m,$$
or equivalently, $\n(A)=m$.

Bearing in mind that $\k(A) \cap \k(B) =\{0\}$ is a necessary condition for invertibility of $\fb{\hat{\bar{\mathcal{K}}}}$, then, if in addition $\ra B=m$, we have
\begin{equation}\label{direct}
	\k(A) \oplus \k(B) = \mathbb{R}^{n}. 
\end{equation}

\begin{theorem}\label{thm1.2}
Let $\mathcal{K}$ be given by \eqref{eq:coef}. Suppose that $A\succeq 0$, $D\succeq 0$ and $E\succeq 0$ such that $\k(A) \cap \k(B)=\{0\}$, $\k(C^T) \cap \k(E) =\{0\}$, and condition \eqref{ker3} \rt{holds}. 
If 
\begin{equation}\label{rang}
\rg(B) \cap \rg(C^T)=\{0\},
\end{equation}
then $\mathcal{K}$ is invertible. Furthermore, condition~\eqref{rang} is necessary 
for the invertibility of $\mathcal{K}$ when \eqref{direct} \rt{is} satisfied and
		\begin{equation}\label{direct2}
	\k(E) \oplus \k(C^T) = \mathbb{R}^{p}.
\end{equation} 
\end{theorem}

\begin{proof}
	Let \eqref{rang} hold. We can conclude the nonsingularity from 
the	\fb{third statement}  in Theorem \ref{thm1}.

To prove the second assertion stated in the theorem, let $\mathcal{K}$ be invertible. Assume that, in contradiction to \eqref{rang}, there exists a nonzero vector $w\in \rg(B) \cap \rg(C^T)$. As a result
\begin{equation}\label{eq1.10}
	w=Bx \quad \text{and} \quad w=C^Tz
\end{equation}
for some $x\in \mathbb{R}^{n}$ and $z\in \mathbb{R}^{p}$. By the assumptions \eqref{direct} and \eqref{direct2}, we deduce that the vectors $x$ and $z$ can be uniquely written in the \ro{form} $x=x_1+x_2$ and $z=z_1+z_2$ such that $x_1\in \k(A), x_2\in \k(B), z_1\in \k(E)$, and $z_2\in \k(C^T)$. From \eqref{eq1.10} it follows that
$w=Bx_1$ and $w=C^Tz_1$.
Considering these two last relations, we conclude that $\mathcal{K} \tilde{u}=0$ for the nonzero vector $\tilde{u}=[x_1;0;-z_1]$, which is a contradiction to the assumed nonsingularity of $\mathcal{K}$. Hence, the subspace $\rg(B) \cap \rg(C^T)$  is trivial, as required.
\end{proof}

The following two theorems reveal that the symmetric positive semidefiniteness requirement of two of the block diagonal matrices in Theorem \ref{thm1.2} can be relaxed  under certain assumptions on ranks and dimensions of $B$ and $C$.

\begin{theorem}\label{thm1.2n}
Let $\mathcal{K}$ be given by \eqref{eq:coef} such that $\k(C^T) \cap \k(E) =\{0\}$. Suppose further that \ro{$n\ge m$}, $\ra B=m$, and \eqref{direct} \rt{holds}. Assume that 
	$A\succeq 0$ and condition \eqref{rang} holds, then $\mathcal{K}$ is invertible.
	When $E$ is zero, condition \eqref{rang}
	is necessary for the invertibility of $\mathcal{K}$.
\end{theorem}

\begin{proof}
	\ro{We first consider the case where $n>m$.} Let \eqref{rang} be satisfied and suppose \rt{that}  $\mathcal{K}\bar u=0$ where $\bar u=[x;y;z]$. Hence, the relations \eqref{hom} are satisfied.
	Let $Z$ be an $n\times (n-m)$ matrix whose columns form a basis for $\k(B)$. \ro{It is known that $Z^TAZ$ is nonsingular, see \cite{G1985}}. In view of \eqref{direct},
	the vector $x$ can be written as $x=x_1+x_2$ where $x_1\in \k(A)$ and $x_2\in \k(B)$. We can write $x_2=Z \tilde{x}_2$ for some $\tilde{x}_2\in \mathbb{R}^{(n-m)}$.	
	Using  Eq.\  \eqref{hom1}, one observes that
\begin{equation}\label{eq18}
		AZ\tilde{x}_2+B^Ty=0.
\end{equation}
Multiplying \eqref{eq18} by $Z^T$ from the left, by nonsingularity of $Z^TAZ$ \ro{and $Z^TB^T=(BZ)^T=0$}, we conclude that $\tilde{x}_2$ is zero, which implies $x_2$ is a zero vector. 
Consequently,  \eqref{eq18} reduces to 
\[
B^Ty=0,
\]
and this, together with $\ra{B}=m$, shows that $y$ is zero. As a result,  Eqs.\  \eqref{hom2} and \eqref{hom3} take the form
\begin{subequations}
\begin{eqnarray}
	Bx_1+C^Tz & =& 0 \label{eq19a} \\
	Ez & =& 0\label{eq19b}.
\end{eqnarray}
\end{subequations}
Since the intersection of $\rg(B)$ and $\rg(C^T)$ is trivial,  \eqref{eq19a} implies that $x_1\in \k(B)$ and $z\in \k(C^T)$. Notice that $z\in \k(E)$ by \eqref{eq19b}, which yields $z\in \k(C^T) \cap \k(E)$. This, together with the fact that $x_1\in \k(A)$, implies that  $z$ and $x$ are zero. It has been already shown that $y$ is zero. Consequently, we \rt{deduce} the nonsingularity of $\mathcal{K}$.

\ro{When $n=m$, by the assumptions 
	$\ra B=n$ and \eqref{direct}, we conclude that $\n{(B)}=0$ and $\n{(A)}=n$. This case happens when $A$ is zero. Consequently, if  $\mathcal{K}\bar u=0$ for $\bar u=[x;y;z]$ then we can immediately observe that $y$ is zero. Similar to the reasoning given above, we can further verify $x$ and $z$ are zero vectors which ensures the nonsingularity of $\mathcal{K}$.
} 

Now, suppose that $E$ is zero and $\mathcal{K}^{-1}$ exists. Let $w\in \rg(B) \cap \rg(C^T)$. Therefore, $w=Bx$  and $w=C^Tz$ for some $x\in \mathbb{R}^{n}$ and $z\in \mathbb{R}^{p}$. By \eqref{direct}, we have $x=x_1+x_2$ such that $x_1\in \k(A)$ and $x_2\in \k(B)$, which implies that $w=Bx_1$. Now it can be seen that $\ro{\mathcal{K}}\bar{u}=0$ for $\bar{u}=[-x_1;0;z]$. The nonsingularity of $\mathcal{K}$ implies that $\bar{u}$ is zero. Hence, the vector $w$ is zero.
\end{proof}

The proof of the following theorem follows \ro{from applying Theorem \ref{thm1.2n}
to the matrix
\begin{equation}\label{sim}
\mathcal{K}_s :=	\left[ {\begin{array}{*{20}{c}}
			E&C&0\\
			{{C^T}}&{ - D}&B\\
			0&{{B^T}}&A
	\end{array}} \right],
\end{equation}
which is similar to $\mathcal{K}$, i.e., $\mathcal{K}_s=\mathcal{P}\mathcal{K}\mathcal{P}$ where $\mathcal{P}$ is the symmetric permutation matrix given as follows:
\[\mathcal{P}= \left[ {\begin{array}{*{20}{c}}
		0&0&I\\
		0&I&0\\
		I&0&0
\end{array}} \right].\]
}

\begin{theorem}\label{thm1.3n}
Let $\mathcal{K}$ be given by \eqref{eq:coef} such that $\k(A) \cap \k(B) =\{0\}$. Furthermore, assume that \ro{$p\ge m$} and $\ra C=m$, condition \eqref{direct2} is satisfied and
$E\succeq 0$.
	If condition \eqref{rang} holds, 
	then   $\mathcal{K}$ is invertible.  When  $A$ is zero, condition \eqref{rang} is necessary for the invertibility of $\mathcal{K}$.
\end{theorem}

\section{Invertibility when the  $(1,1)$-block is maximally rank deficient}\label{sec3}

In this section we \fb{mainly focus on obtaining the necessary and sufficient conditions for the existence of $\mathcal{K}^{-1}$} defined in \eqref{eq:coef}  when $\n(A)=m$. This case is particularly interesting in applications related to electromagnetics, such as time-harmonic Maxwell's equations and incompressible magnetohydrodynamics problems. In those cases the leading block is a discrete curl-curl operator, which is known to have a large kernel of gradient functions.

\ro{As previously mentioned, the matrix $\mathcal{K}$ is similar to $\mathcal{K}_s$ given in \eqref{sim}. Consequently, the following established results can be stated for $\mathcal{K}_s$ which results in a distinct set of assumptions on the blocks. This entails swapping the roles of $A$, $B$, and $C$ with $E$, $C^T$, and $B^T$, respectively. 
}

\subsection{On the nullity of the $(3,3)$ diagonal block}

We establish a connection between the invertibility of the matrix $E$, which is the $(3,3)$ diagonal block of $\mathcal{K}$, and the invertibility of $\mathcal{K}$. Some additional connections between the nullity of $E$ and the nullity of other blocks of $\mathcal{K}$ or its inverse are then provided.
Recall the following useful theorem. 
\begin{theorem}\label{th3.5EG}\cite[Theorem 3.5]{EG2015}
	Suppose that $\n(A)=m$, $\k(A) \cap \k(B) = \{0\}$ and let $W\in \mathbb{R}^{m\times m}$ be an invertible matrix. Then,
	\[B(A+B^TW^{-1}\rt{B})^{-1}B^T=W.\]
\end{theorem}
We can use the result stated in Theorem \ref{th3.5EG} to establish additional necessary and sufficient conditions  for the invertibilty of $\mathcal{K}$, as follows.
\begin{proposition}\label{prop2.2}
	Let $A\succeq 0$, $D\succeq 0$ and assume conditions $(i)$--$(iii)$ in Proposition \ref{prop1.1} hold. Assume also that $\n(A)=m$. Then, the matrix $\mathcal{K}$ is invertible if and only if the matrix 
		$\tilde{\mathcal{S}}$ defined below is invertible:\rt{
	\begin{eqnarray}
		\tilde{\mathcal{S}} &= &\left[ {\begin{array}{*{20}{c}}
				{ - \frac{1}{\alpha }{{(2I - \alpha D)}^{ - 1}}}&{{{(2I - \alpha D)}^{ - 1}}{C^T}}\\
				{C{{(2I - \alpha D)}^{ - 1}}}&{E - \alpha C{{(2I - \alpha D)}^{ - 1}}{C^T}}
		\end{array}} \right],  \label{Sc2}
	\end{eqnarray}}
where $\alpha$ is a scalar that satisfies, for a nonzero matrix $D$,
	\begin{equation}\label{alpha}
		0< \alpha < \frac{2}{\lambda_{\max}(D)}.
	\end{equation}
\end{proposition}

\begin{proof}
	Let us first define
	\begin{equation}\label{matrixW}
		\mathcal{W} = \left[ {\begin{array}{*{20}{c}}
				I&0&0\\
				\alpha B&I&0\\
				0&0&I
		\end{array}} \right].
	\end{equation}
	Consider the matrix
\begin{equation}\label{re1}
		{\mathcal{W}^T}	\mathcal{K}\mathcal{W} = \tilde{\mathcal{K}},
\end{equation}
	where 
	\begin{equation}\label{Cmatrix}
		\tilde{\mathcal{K}}=\left[ {\begin{array}{*{20}{c}}
				{A  + \alpha {B^T}(2I - \alpha D)B}&{{{(B - \alpha DB)}^T}}&{{{\alpha(CB)}^T}}\\
				{B - \alpha DB}&{ - D}&{{C^T}}\\
				{\alpha CB}&C&E
		\end{array}} \right].
	\end{equation}
	Notice that $2I-\alpha D \succ 0$, so the block $A  + \alpha {B^T}(2I - \alpha D)B \succ 0$. Using Theorem \ref{th3.5EG} with \ro{$W=\frac{1}{\alpha}(2I - \alpha D)^{-1}$}, we can verify that
\begin{equation}\label{s5}
		\tilde{\mathcal{K}}= \left[ {\begin{array}{*{20}{c}}
			I&{0}\\
			{\mathcal{B}{\tilde{A}^{-1}}}&I
	\end{array}} \right]\left[ {\begin{array}{*{20}{c}}
			{\tilde{A}}&{0}\\
			{0}&\tilde{\mathcal{S}}
	\end{array}} \right]\left[ {\begin{array}{*{20}{c}}
			I&{\tilde{A}^{-1}\mathcal{B}^T}\\
			{0}&I
	\end{array}} \right],
\end{equation}
	where $\tilde{A}:=A  + \alpha {B^T}(2I - \alpha D)B$,  $\mathcal{B}:=[B - \alpha DB;\alpha CB]$ and 
\begin{eqnarray}
\nonumber	\tilde{\mathcal{S}} & = &	\left[ {\begin{array}{*{20}{c}}
			{ - D}&{{C^T}}\\
			C&E
	\end{array}} \right] - \left[ {\begin{array}{*{20}{c}}
			{\frac{1}{\alpha }(I - \alpha D){{(2I - \alpha D)}^{ - 1}}(I - \alpha D)}&{(I - \alpha D){{(2I - \alpha D)}^{ - 1}}{C^T}}\\
			{C{{(2I - \alpha D)}^{ - 1}}(I - \alpha D)}&{\alpha C{{(2I - \alpha D)}^{ - 1}}{C^T}}
	\end{array}} \right].
\end{eqnarray}
\rt{Denoting $M=2I-\alpha D$, one can observe that
\begin{eqnarray}
	\nonumber 	\tilde{\mathcal{S}} & = & \left[ {\begin{array}{*{20}{c}}
			{ - D}&{{C^T}}\\
			C&E
	\end{array}} \right] - \left[ {\begin{array}{*{20}{c}}
			{\frac{1}{\alpha }(M - I){M^{ - 1}}(M - I)}&{(M - I){M^{ - 1}}{C^T}}\\
			{C{M^{ - 1}}(M - I)}&{\alpha C{M^{ - 1}}{C^T}}
	\end{array}} \right]\\
	\nonumber &&\\
	\nonumber & = &\left[ {\begin{array}{*{20}{c}}
		{ - D}&{{C^T}}\\
		C&E
\end{array}} \right] - \left[ {\begin{array}{*{20}{c}}
		{\frac{1}{\alpha }(M - 2I + {M^{ - 1}})}&{(M - I){M^{ - 1}}{C^T}}\\
		{C{M^{ - 1}}(M - I)}&{\alpha C{M^{ - 1}}{C^T}}
\end{array}} \right]\\
	\nonumber &&\\
	\nonumber & = &\left[ {\begin{array}{*{20}{c}}
		{ - D}&{{C^T}}\\
		C&E
\end{array}} \right] - \left[ {\begin{array}{*{20}{c}}
		{ - D + \frac{1}{\alpha }{M^{ - 1}}}&{{C^T} - {M^{ - 1}}{C^T}}\\
		{C - C{M^{ - 1}}}&{\alpha C{M^{ - 1}}{C^T}}
\end{array}} \right]\\
	\nonumber &&\\
\nonumber & = &\left[ {\begin{array}{*{20}{c}}
		{ - \frac{1}{\alpha }{M^{ - 1}}}&{{M^{ - 1}}{C^T}}\\
		{C{M^{ - 1}}}&{E - \alpha C{M^{ - 1}}{C^T}}
\end{array}} \right].
\end{eqnarray}	
}
Now it is immediate to deduce \eqref{Sc2}. 
\end{proof}

In practice, it is difficult to check the invertibility of 
$\tilde{\mathcal{S}}$ in Proposition \ref{prop2.2}. However, it turns out that $\tilde{\mathcal{S}}$ can be efficiently factored  when \rt{$\lambda_{\max}{(D)}<2$.} 

\begin{proposition}\label{prop3.2}
Assume the same conditions as in Proposition \ref{prop2.2}, but \ro{with the additional assumption}
\rt{$\lambda_{\max}{(D)}<2$}. Then, the matrix $E$ is nonsingular if and only if $\mathcal{K}$ 
is nonsingular. 
\end{proposition}
\begin{proof}
	Using the same notation and quantities as in Proposition \ref{prop2.2}, \rt{we set $\alpha=1$ and it is immediate to observe that the condition of the proposition is fulfilled. Denoting $\bar{C}=C(2I-D)^{-1/2}$, we can verify that
	\[\tilde{\mathcal{S}}=\left[ {\begin{array}{*{20}{c}}
			{{{(2I - D)}^{ - \frac{1}{2}}}}&0\\
			0&I
	\end{array}} \right]\left[ {\begin{array}{*{20}{c}}
			{ - I}&{{\bar{C}^T}}\\
			\bar{C}&{E - \bar{C}{\bar{C}^T}}
	\end{array}} \right]\left[ {\begin{array}{*{20}{c}}
			{{{(2I - D)}^{ - \frac{1}{2}}}}&0\\
			0&I
	\end{array}} \right].\]
Hence, by Proposition \ref{prop2.2}, $\tilde{\mathcal{S}}$ is invertible if and only if the matrix
	\[\left[ {\begin{array}{*{20}{c}}
			{ - I}&{{\bar{C}^T}}\\
			\bar{C}&{E - \bar{C}{\bar{C}^T}}
	\end{array}} \right]\]
 is invertible. Straightforward algebraic computations reveal that 
\begin{equation}\label{decom1}
\left[ {\begin{array}{*{20}{c}}
		{ - I}&{{\bar{C}^T}}\\
		\bar{C}&{E - \bar{C}{\bar{C}^T}}
\end{array}} \right]= \left[ {\begin{array}{*{20}{c}}
			I&0\\
			{ - \bar{C}}&I
	\end{array}} \right]\left[ {\begin{array}{*{20}{c}}
			{ -I}&0\\
			0&E
	\end{array}} \right]\left[ {\begin{array}{*{20}{c}}
			I&{ - {\bar{C}^T}}\\
			0&I
	\end{array}} \right].
\end{equation}}
Now it is immediate to conclude the assertion. 
\end{proof}

\rt{By Proposition \ref{prop3.2}, setting $\alpha=1$ and using decomposition  \eqref{s5}, we can observe that the matrix $\tilde{\mathcal{K}}$ in \eqref{Cmatrix} can be written as follows:
	\begin{equation}\label{eq:cong}
		\tilde{\mathcal{K}}=	\left[ {\begin{array}{*{20}{c}}
				I&0&0\\
				{{B_1}\tilde{A}_1^{ - 1}}&I&0\\
				{CB\tilde{A}_1^{ - 1}}&{ - C}&I
		\end{array}} \right]\left[ {\begin{array}{*{20}{c}}
				{{\tilde{A}_1}}&0&0\\
				0&{ - {{(2I - D)}^{ - 1}}}&0\\
				0&0&E
		\end{array}} \right]\left[ {\begin{array}{*{20}{c}}
				I&{\tilde{A}_1^{ - 1}B_1^T}&{\tilde{A}_1^{ - 1}{B^TC^T}}\\
				0&I&{ - {C^T}}\\
				0&0&I
		\end{array}} \right],
	\end{equation}
where $\tilde{A}_1=A  +  {B^T}(2I - D)B$ and $B_1=B - DB$.
Hence, provided that $E$ is nonsingular, the inverse of $\tilde{\mathcal{K}}$ exists and it can be decomposed as
\begin{equation}\label{invdec}
	\tilde{\mathcal{K}}^{-1}=\left[ {\begin{array}{*{20}{c}}
			I&{ - {{\tilde{A}_1}^{ - 1}}{B_1^T}}&{ - {{\tilde{A}_1}^{ - 1}}{(B+B_1)^T}{C^T}}\\
			0&I&{{C^T}}\\
			0&0&I
	\end{array}} \right]\left[ {\begin{array}{*{20}{c}}
			{{{\tilde{A}_1}^{ - 1}}}&0&0\\
			0&{ -(2I-D)}&0\\
			0&0&{{E^{ - 1}}}
	\end{array}} \right]\left[ {\begin{array}{*{20}{c}}
			I&0&0\\
			{ - B_1{{\tilde{A}_1}^{ - 1}}}&I&0\\
			{ - C(B+B_1){{\tilde{A}_1}^{ - 1}}}&C&I
	\end{array}} \right].
\end{equation}}

\rt{It is evident that if we add {$\lambda_{\max}(D)<2$} to the assumptions of Proposition \ref{prop2.2},} then nonsingularity of $E$ is a necessary condition for the existence of $\mathcal{K}^{-1}$. As observed in the previous section, the existence of $E^{-1}$ is not always necessary for the nonsingularity of $\mathcal{K}$.
In the following theorem, we assume that $\mathcal{K}^{-1}$ exists and derive some relations between the nullity of the second block diagonal of  $\mathcal{K}^{-1}$ and $\n(A)$ and $\n(E)$. The proof of the theorem is inspired by \cite[Theorem 2.1]{SN2004}.

\begin{theorem}\label{thm3.1}
	Let $\mathcal{K}$ be invertible with the dimensions $n, m,$ and $p$ defined \eqref{eq:coef}, and
	consider the following partitioning of the inverse,
	\begin{equation}\label{strucinv}
		\mathcal{K}^{-1}=\left[ {\begin{array}{*{20}{c}}
				{{Z_{11}}}&{{Z_{12}}}&{{Z_{13}}}\\
				{{Z_{21}}}&{{Z_{22}}}&{{Z_{23}}}\\
				{{Z_{31}}}&{{Z_{32}}}&{{Z_{33}}}
		\end{array}} \right],	
	\end{equation}
	where $Z_{11},Z_{22}$ and $Z_{33}$ are square matrices with dimensions $n, m,$ and $p$, respectively. Then,
		\begin{equation}\label{eq22n}
		\min\{\max\{\n(A),\n(E)\},m\} \le	\n(Z_{22})\le \n(A)+\n(E).
	\end{equation}
In addition, if  condition \eqref{rang} is satisfied, then
	\begin{equation}\label{eq22}
	\min\{\n(A)+\n(E),m\} \le	\n(Z_{22})\le \n(A)+\n(E).
	\end{equation}

\end{theorem}
\begin{proof}	
	Given a matrix $W$,  let $N(W)$ denote a matrix whose columns form a basis for $\k(W)$. In fact, the number of columns of  $N(W)$ is the nullity of $W$.
	To verify relations \eqref{eq22n} and \eqref{eq22}, we use $\mathcal{K} \mathcal{K}^{-1}=\mathcal{K}^{-1}\mathcal{K}=I$. First, note that
	\begin{eqnarray*}
		0=(\mathcal{K}^{-1}\mathcal{K})_{21} &=& Z_{21}A+Z_{22}B \\
		0=(\mathcal{K}^{-1}\mathcal{K})_{23}  &=& Z_{22} C^T+Z_{23}E.
	\end{eqnarray*}
Consequently, we get
	\begin{equation}\label{eq2.2}
		Z_{22}[BN(A) \quad  C^TN(E)]=0.
	\end{equation}
Note that $BN(A)$ and $C^TN(E)$ have full \ro{column rank.} Indeed, if there exist $y_1$ and $y_2$ such that 	\[
	BN(A)y_1 =0 \quad \text{and} \quad C^TN(E)y_2=0,
	\]
	then
\begin{equation}\label{con1}
		N(A)y_1 \in \k(A)\cap \k(B)\quad \text{and} \quad N(E)y_2 \in \k(C^T)\cap \k(E).
\end{equation}
		Since $\mathcal{K}$ is nonsingular, by Proposition \ref{prop1.1}, the above relations yield  
	\[
	N(A)y_1 =0 \quad \text{and} \quad N(E)y_2=0,
	\]
	which ensures that $y_1=0$ and $y_2=0$. Therefore, we conclude
\begin{equation}\label{con2}
		\n(Z_{22}) \ge \min\{\n(A),m\} \quad \text{and} \quad 	\n(Z_{22}) \ge \min\{\n(E),m\}.
\end{equation}
If \eqref{rang} is satisfied, we show that the columns of  $[BN(A) \quad  C^TN(E)]$ are linearly independent.
	To this end, let the vector $y=[y_1;y_2]$ \ro{be such that} 
	\[
	BN(A)y_1+ C^TN(E)y_2=0.
	\] 
	The above relation together with \eqref{rang} imply that
	\[
	BN(A)y_1 =0 \quad \text{and} \quad C^TN(E)y_2=0,
	\]
	which leads to \eqref{con1}.
Hence, we deduce that the vectors $y_1$ and $y_2$ are both zero, and \eqref{eq2.2} implies that
	\begin{equation}\label{eq2.3}
		\n(Z_{22}) \ge \min\{\n(A)+\n(E),m\}.
	\end{equation}
	Using the following identities
	\begin{eqnarray*}
		0 = (\mathcal{K}\mathcal{K}^{-1})_{12} &=& AZ_{12}+B^TZ_{22}  \\
		0 = (\mathcal{K}\mathcal{K}^{-1})_{32} &=& CZ_{22}+EZ_{32},
	\end{eqnarray*}
	we find
	\[
	AZ_{12}N(Z_{22})=0 \quad \text{and} \quad EZ_{32}N(Z_{22})=0,
	\]
	which is equivalent to saying that
	\begin{equation}\label{eq2.4}
		\left[ {\begin{array}{*{20}{c}}
				A&0\\
				0&E
		\end{array}} \right]\left[ {\begin{array}{*{20}{c}}
				{{Z_{12}}N(Z_{22})}\\
				{{Z_{32}}N(Z_{22})}
		\end{array}} \right] = \left[ {\begin{array}{*{20}{c}}
				0\\
				0
		\end{array}} \right].
	\end{equation}
	In the sequel, we first show that the columns of $[{{Z_{12}}N(Z_{22})};{{Z_{32}}N(Z_{22})}]$
	are linearly independent. To do so, let 
	\[
	\left[ {\begin{array}{*{20}{c}}
			{{Z_{12}}N(Z_{22})}\\
			{{Z_{32}}N(Z_{22})}
	\end{array}} \right]y=0.
	\]
	As a result, we have ${{Z_{12}}N(Z_{22})}y=0$ and ${Z_{32}}N(Z_{22})y=0$. Therefore, bearing in mind that ${Z_{22}}N({Z_{22}})$ is zero, we conclude that
	\begin{eqnarray*}
		\mathcal{K}^{-1} \left[ {\begin{array}{*{20}{c}}
				0\\
				{N({Z_{22}})y}\\
				0
		\end{array}} \right] = \left[ {\begin{array}{*{20}{c}}
				{{Z_{12}}N({Z_{22}})y}\\
				{{Z_{22}}N({Z_{22}})y}\\
				{{Z_{32}}N({Z_{22}})y}
		\end{array}} \right]=\left[ {\begin{array}{*{20}{c}}
				0\\
				0\\
				0
		\end{array}} \right].
	\end{eqnarray*}
	
	From the above relation, it is immediate to conclude that $N(Z_{22})y=0$, which implies $y=0$. By \eqref{eq2.4}, we have
	\[
	\n(A)+\n(E) \ge \n(Z_{22}).
	\]
The above relation together with \eqref{con2} and \eqref{eq2.3} shows that both \eqref{eq22n} and \eqref{eq22} hold.
\end{proof}

We end this part by commenting that if $\n(E)=0$, regardless of condition \eqref{rang}, the following \ro{relations} hold
\[
\min\{\n(A),m\} \le	\n(Z_{22})\le \n(A),
\]
by Theorem \ref{thm3.1}. 	In particular, if $\n(A)=m$  then  $Z_{22}$ is the zero matrix. 

\subsection{An explicit formula for the inverse}

{As pointed out in the previous subsection, the nonsingularity of $\mathcal{K}$ implies the existence of $E^{-1}$ under certain conditions. In addition, by Theorem \ref{thm3.1},  the second block of $\mathcal{K}^{-1}$ \ro{is} zero when  $\n(A)=m$ and $\n(E)=0$.
 We now assume that $E$ is nonsingular and derive an explicit formula for the inverse of $\mathcal{K}$ without imposing any restrictions on  $D$.   Define
\begin{equation}
V=Z(Z^TAZ)^{-1}Z^T,
\label{eq:V}
\end{equation}
where $Z$ is a matrix whose columns form a basis for $\k(B)$. In the context of constrained optimization, the matrix $Z^T A Z$ is known as the reduced Hessian and it plays an important role in null-space methods \cite{GMSW1991}. Here we assume that $A\succeq 0$ and $\k(A) \cap \k(B) =\{0\}$, which ensures the nonsingularity of $Z^TAZ$ \cite{G1985}.

We start by establishing the following two propositions that facilitate the derivation of a 
formula for $\mathcal{K}^{-1}$. 
\begin{proposition}\label{prop2.1}
Let  $A\succeq 0$ with $\n{(A)}=m$, and suppose condition \eqref{direct} holds. Then
$$
A=AVA, 
$$
where $V$ is defined in \eqref{eq:V}.
\end{proposition}
\begin{proof}
Let $Z$ be a matrix whose columns form a basis for $\k(B)$ and consider Eq.\  \eqref{eq:V}. We prove the desired result by verifying that for any vector $x$,
\begin{equation}\label{eq2.6}
	Ax=AVAx.
\end{equation}
We have 
\begin{equation}
\begin{aligned}
	AVAZ & = AZ(Z^TAZ)^{-1}Z^TAZ\\
	& = AZ. \label{eq2.5}
\end{aligned}
\end{equation}
 Given an arbitrary vector $x$, by \eqref{direct} we can write  $x=x_1+x_2$ where $x_1\in \k(A)$ and $x_2\in \k(B)$. Trivially, $A x_1=0$, and hence, we  need to show the validity of  \eqref{eq2.6} for $x_2\in \k(B)$. We can write $x_2=Z\tilde{x}$ for some $\tilde{x}$. Consequently,  \eqref{eq2.6} can be rewritten as 
	\[
	AZ\tilde{x}=AVAZ\tilde{x},
	\]
	which completes the proof, using \eqref{eq2.5}.
\end{proof}

Suppose $Z$ is a matrix whose columns form an orthonormal basis for $\k(B)$. This ensures that
\begin{equation}
	{B^T}{(B{B^T})^{ - 1}}B = I - Z{Z^T}.
\label{eq:Z}
\end{equation}
Relation \eqref{eq:Z} will come handy in the proofs of Theorems \ref{th2.3} and \ref{th2.3n}. 

In \cite{EG2015} an explicit formula for the inverse \ro{of a classical two-by-two saddle point system} is derived, which shows that if $\n{(A)}=m$ and the trailing main (2,2) block is zero, then the inverse has a trailing zero block as well. The existence of the trailing zero block can also be established by the rank relations analyzed in \cite{SN2004}. In \cite{WG2020} the nonzero structure of the inverse of a matrix for an incompressible magnetohydrodynamics model problem is used to design a sparse approximate inverse as a preconditioner.

Below we show that the trailing block of the inverse remains zero when the trailing block  of the matrix is nonzero.
\begin{theorem}\label{th2.3}
Suppose $A\succeq 0$ with  $\n{(A)}=m$, and assume condition \eqref{direct} holds. Then 
	\[
	 \hat{\mathcal{K}}:={\left[ {\begin{array}{*{20}{c}}
				A&{{B^T}}\\
				B&{ - D}
		\end{array}} \right]}
	\]
is invertible and its inverse is given by 
\begin{equation}\label{inv}
		 \hat{\mathcal{K}}^{-1}=\left[ {\begin{array}{*{20}{c}}
			{  (I - VA){B^T}{{(B{B^T})}^{ - 1}}D{{(B{B^T})}^{ - 1}}B(I - AV) + V}&{}&{(I - VA){B^T}{{(B{B^T})}^{ - 1}}}\\
			{{{(B{B^T})}^{ - 1}}B(I - AV)}&{}&0
	\end{array}} \right],
\end{equation}
where $V$ is as  in \eqref{eq:V} with $Z\in \mathbb{R}^{n \times (n-m)}$ being any matrix whose columns form an orthonormal basis for $\k(B)$.
\end{theorem}

\begin{proof}
Assume that $\hat{\mathcal{K}}\hat{u}=0$ where $\hat{u}=[x;y]$.
	Because of \eqref{direct}, we have $x=x_1+x_2$ where $x_1\in \k(A)$ and $x_2\in \k(B)$. Notice that one can verify that $x_2=Z\hat{x}_2$ for some $\hat{x}_2\in \mathbb{R}^{(n-m)}$.
As a result, we get 
	\begin{subequations}
	\begin{eqnarray}
		AZ\hat{x}_2+B^Ty & = & 0 \label{2.5a}\\
		Bx_1-Dy & = & 0. \label{2.5b}
	\end{eqnarray}\end{subequations}
	From here we can proceed similarly to the way null-space methods are derived \cite{GMSW1991}. 
Multiplying Eq.\  \eqref{2.5a} by $Z^T$ from the left and using $Z^T B^T =0$, we obtain $Z^T A Z \hat{x}_2$ and conclude that $\hat{x}_2$ is zero invoking the fact that $Z^TAZ$ is invertible.  Also, we can observe that $y$ is zero from Eq.\  \eqref{2.5a} and the fact that  the columns of $B^T$ are linearly independent. From \eqref{2.5b}, we can deduce that $x_1=0$, which establishes the nonsingularity of $\hat{\mathcal{K}}$.

From \eqref{eq:V} it follows that
\begin{equation}\label{VAZZsym}
	ZZ^T AV=ZZ^T A Z(Z^TAZ)^{-1}Z^T =ZZ^T.
\end{equation}
Using  \eqref{eq:Z}, \eqref{VAZZsym} and Proposition \ref{prop2.1} yields $A(I-VA)=(I-AV)A=0$. It can now be readily verified that  $\hat{\mathcal{K}}^{-1}\hat{\mathcal{K}}=I$ where $\hat{\mathcal{K}}^{-1}$ is given by \eqref{inv}, as required, 
\end{proof}

 Under certain conditions, we can further derive an explicit formula
for the inverse of $\mathcal{K}$ given by \eqref{eq:coef}. To this end, we present the following theorem.

\begin{theorem}\label{th2.3n}
Suppose that  $A\succeq 0$  with $\n{(A)}=m$ and condition \eqref{direct} holds. If the matrix $E$ is symmetric and nonsingular, then the inverse of $\mathcal{K}$ is given by
		\begin{equation}\label{inv2}
			{\mathcal{K}}^{-1}=\left[ {\begin{array}{*{20}{c}}
					T&{{R^T}}&{{S^T}}\\
					R&0&0\\
					S&0&{{E^{ - 1}}}
			\end{array}} \right],
		\end{equation}
		where
		\begin{eqnarray*}
			T&:=&{(I - VA){B^T}{{(B{B^T})}^{ - 1}}\left( {D + {C^T}{E^{ - 1}}C} \right){{(B{B^T})}^{ - 1}}B(I - AV)}+V\\
			R&:=&{{{(B{B^T})}^{ - 1}}B(I - AV)}\\
			S&:=&{ - {E^{ - 1}}C{{(B{B^T})}^{ - 1}}B(I - AV)},
		\end{eqnarray*}
		and $V$ is as defined in \eqref{eq:V}, where $Z\in \mathbb{R}^{n \times (n-m)}$ is any matrix whose columns form an orthogonal basis for $\k(B)$.
\end{theorem}
\begin{proof}
	\ro{
To conclude the assertion, we verify that $	{\mathcal{K}}^{-1}\mathcal{K}=I$ where $	{\mathcal{K}}^{-1}$ is given by \eqref{inv2}.
{Using Proposition \ref{prop2.1}, Eqs.\  \eqref{eq:Z} and \eqref{VAZZsym}, we can conclude}
\begin{eqnarray*}
({\mathcal{K}}^{-1}\mathcal{K})_{11}&=&TA+R^TB \\
&=&VA+(I-VA)B^T(BB^T)^{-1}B\\
& = & VA +(I-VA)(I-ZZ^T)\\
&=& I-ZZ^T+VAZZ^T=I.
\end{eqnarray*}
Taking into account that $VB^T=(BV)^T$ is zero, and applying some algebraic computations, we can check that $({\mathcal{K}}^{-1}\mathcal{K})_{12}=0$, $({\mathcal{K}}^{-1}\mathcal{K})_{32}=0$ and $({\mathcal{K}}^{-1}\mathcal{K})_{22}=I$. We can immediately conclude from Proposition \ref{prop2.1} that $({\mathcal{K}}^{-1}\mathcal{K})_{21}=0$ and $({\mathcal{K}}^{-1}\mathcal{K})_{31}=0$. In addition, straightforward algebraic computations reveal that $({\mathcal{K}}^{-1}\mathcal{K})_{13}=0$, $({\mathcal{K}}^{-1}\mathcal{K})_{23}=0$  and $({\mathcal{K}}^{-1}\mathcal{K})_{33}=I$.} 
\end{proof}

\section{Concluding remarks and future work}\label{sec4}
The conditions on invertibility provided in this work may be useful to  understand under what circumstances double saddle-point systems of the form \eqref{eq:system} can be solved. From a theoretical point of view, this is a necessary step in the analysis of solvability and other algebraic properties of such systems. From a numerical standpoint, the formulas of the inverses and their possible decompositions, may be useful within the context of developing preconditioned iterative solvers based on sparse approximate inverses. The fact that some of the blocks of the inverse are zero under appropriate rank conditions is potentially useful for deriving preconditioners with a specific block structure. Such preconditioners may be best utilized if additional information on the underlying application beyond the algebraic structure of the blocks is available. \fb{Specifically, developing efficient preconditioners and analyzing the spectrum of the corresponding preconditioned matrices by exploiting the expressions for the inverse of $\mathcal{K}$ (Eqs.\  \eqref{invdec} and \eqref{inv2}) or the inverse of its two-by-two sub-block (Eq.\  \eqref{inv}) to accelerate Krylov subspace methods is currently under investigation.}

\section*{Acknowledgments}
The first author thanks the hospitality of Simon Fraser University's Department of Mathematics and the University of British Columbia's Department of Computer Science during her visit, where this work was completed. 
The work of the second author was supported by a Natural Sciences and Engineering Research Council (NSERC) of Canada Discovery grant (RGPIN-2023-05244).
The work of the third author was supported by a Natural Sciences and Engineering Research Council (NSERC) of Canada Discovery grant (RGPIN-2020-04663). Our sincere thanks  to two anonymous referees for their helpful suggestions and constructive comments.

\bibliographystyle{abbrvnat}

\bibliography{refs.bib}

\end{document}